\documentclass[12pt]{article}

\usepackage[utf8]{inputenc}

\usepackage[a4paper,nomarginpar]{geometry}
\usepackage[english]{babel}

\usepackage{amsmath,amsfonts,amssymb,amsthm,amscd,mathrsfs}
\usepackage[pdftex]{graphicx}

\AtBeginDocument{
   \def\MR#1{}
}

\theoremstyle{definition}
\newtheorem{theorem}{Theorem}

\newtheorem{proposition}[theorem]{Proposition}
\newtheorem{corollary}[theorem]{Corollary}

\theoremstyle{definition}
\newtheorem{definition}[theorem]{Definition}
\newtheorem{example}[theorem]{Example}

\theoremstyle{remark}

\newcommand{\N}{\mathbb{N}} 
\newcommand{\R}{\mathbb{R}} 
\newcommand{\C}{\mathbb{C}} 

\newcommand{\norm}[1]{\lVert#1\rVert}

\newcommand{\orb}{\operatorname{Orb}}

\newcommand{\ind}{\operatorname{ind}}
\newcommand{\proj}{\operatorname{proj}}

\begin{document}

\title{A hypercyclicity criterion for non-metrizable topological vector spaces}

\author{ Alfred Peris\footnote{e-mail: aperis@mat.upv.es}}

\date{ }

\maketitle

\vspace{5mm}

\begin{center}
 \textbf{Dedicated to the memory of Professor Pawe{\l} Doma\'nski}
\end{center}

\vspace{5mm}

\begin{abstract}
We provide a sufficient condition for an operator $T$ on a non-metrizable and sequentially separable topological vector space $X$ to be sequentially hypercyclic. This condition is applied to some particular examples, namely, a composition operator on the space of real analytic functions on $]0,1[$, which solves two problems of Bonet and Doma\'nski \cite{bd12},  and the ``snake shift'' constructed in \cite{bfpw} on direct sums of sequence spaces. The two examples have in common that they do not admit a densely embedded  F-space $Y$ for which the operator restricted to $Y$ is continuous and hypercyclic, i.e., the hypercyclicity of these operators cannot be a consequence of the comparison principle with hypercyclic operators on F-spaces.
\end{abstract}


The study of the dynamics of linear operators has experienced a great development in recent years, with two monographs \cite{bayart_matheron2009dynamics} and \cite{grosse-erdmann_peris2011linear}, and many research papers. Usually the interest is in the dynamics of (continuous and linear) operators $T\in L(X)$ defined on separable Fréchet spaces $X$. Metrizability and completeness of the space offers the possibility to apply Baire category arguments, which are very useful in this context. A few articles concentrate on the dynamics of operators on non-metrizable topological vector spaces (see, e.g., \cite{bd12,bfpw,shk}).

We recall that an operator $T\in L(X)$ on a topological vector space $X$ is \emph{hypercyclic} if there are $x\in X$ whose orbit $\orb (x,T):=\{ x,Tx,T^2x,\dots \}$ is dense in $X$. We will say that $T$ is \emph{sequentially hypercyclic} if there is $x\in X$ such that, for each $y\in X$, there exists an increasing sequence of integers $(n_k)_k$ such that $\lim_k T^{n_k}x=y$. Also, to avoid confusion with more general concepts, we say that $X$ is \emph{sequentially separable} if there exists a countable set $A\subset X$ such that any $z\in X$ is the limit of a sequence in $A$.

In many cases one obtains (sequential) hypercyclicity of an operator $T\in L(X)$ for a non-metrizable $X$ by finding a Fréchet space $Y$, a hypercyclic operator $S$ on $Y$, and a continuous map $\Psi :Y\to X$ with dense range such that $T\circ \Psi=\Psi\circ S$. This is the so-called \emph{comparison principle}. An exception to this procedure are the hypercyclic operators on non-metrizable topological vector spaces obtained in \cite{bfpw} and \cite{shk}, where the hypercyclic vectors are constructed directly. We will obtain criteria under which operators on general topological vector spaces are sequentially hypercyclic.

\section{Criteria for sequential hypercyclicity}

In this section we will provide useful sufficient conditions for sequential hypercyclicity of operators on (non-metrizable) topological vector spaces.

\begin{proposition} Let $X$ be a sequentially separable topological vector space and $T\in L(X)$ such that there exist a
sequentially dense set $X_0:=\{ x_n \ ; \ n\in \N\}\subset X$, a sequence of maps $S_n:X_0\rightarrow X$, $n\in\N$,
a subspace $Y\subset X$  with a finer topology $\tau$ such that $(Y,\tau)$ is an F-space for which we fix
a countable basis of balanced $0$-neighbourhoods $(V_n)_n$ with $V_n+V_n\subset V_{n-1}$, $n>1$,and an increasing sequence $(n_k)_k$ of natural numbers
($n_0:=0$) satisfying:
\begin{enumerate}
\item[(i)] $T^{n_k}S_{n_j}x_j\in V_{2k}$, $k>1$, $j=1,\dots ,k-1$,
\item[(ii)] $T^{n_k}S_{n_j}x_j\in V_{j}$, $k\geq 0$, $j>k$,
\item[(iii)]  $x_k-T^{n_k}S_{n_k}x_k\in V_k$, $k\in \N$.
\end{enumerate}
Then $T$ is sequentially hypercyclic.
\end{proposition}

\begin{proof}  Let
$$
x:=\sum_{j=1}^\infty S_{n_j}x_j ,
$$
which belongs to $Y$ by (ii) for $k=0$, since $Y$ is an F-space. Conditions (i), (ii) and (iii) yield
that
$$
x_k-T^{n_k}x=-\left(T^{n_k}\left( \sum_{j=1}^{k-1} S_{n_j}x_j\right)\right)+\left(x_k-T^{n_k}S_{n_k}x_k\right) -
\left(T^{n_k}\left( \sum_{j=k+1}^{\infty} S_{n_j}x_j\right)\right) \in V_{k-2},
$$
for all $k\geq 2$, and we conclude that $T$ is sequentially hypercyclic.
\end{proof}

Actually, to apply this criterion in some particular examples, we will use other conditions which are stronger, but easy to verify.

\begin{definition}
We say that a sequence $(x_j)_j$ is \emph{eventually contained} in a set $A$ (denoted by $(x_j)_j \subset_{ec} A$) if there is
an integer $j_0$ such that $x_j\in A$ for $j\geq j_0$.
\end{definition}

\begin{corollary}\label{mcor}    Let $X$ be a sequentially separable topological vector space and $T\in L(X)$ such that there exist a
sequentially dense set $X_0:=\{ x_n \ ; \ n\in \N\}\subset X$, a sequence of maps $S_n:X_0\rightarrow X$, $n\in\N$,
a subspace $Y\subset X$  with a finer topology $\tau$ such that $(Y,\tau)$ is an F-space, and an increasing sequence $(n_k)_k$ of natural numbers ($n_0:=0$) satisfying:
\begin{enumerate}
\item[(i)']  $(T^{n_k}S_{n_j}x)_k\subset_{ec}Y$ and converges to $0$ in $(Y,\tau )$ for each $x\in X_0$, and for all $j\in \N$,
\item[(ii)'] $(T^{n_j}S_{n_k}x)_k\subset_{ec}Y$ and converges to $0$ in $(Y,\tau )$ for each $x\in X_0$, and for all $j\geq 0$,
\item[(iii)'] $(x-T^{n_k}S_{n_k}x)_k\subset_{ec}Y$ and converges to $0$ in $(Y,\tau )$ for each $x\in X_0$.
\end{enumerate}
Then $T$ is sequentially hypercyclic.
\end{corollary}

\section{Composition operators on the space of real analytic functions and shifts on direct sums}

In this section we will apply the previous criterion to a composition operator on the space of real analytic functions
$\mathscr{A}(]0,1[)$, solving two
questions of Bonet and Doma\'nski in \cite{bd12}, and to the ``snake shift'' constructed in \cite{bfpw} on countable
direct sums of sequence spaces.

We first recall some basic definitions on the spaces of real analytic functions and composition operators between them. Given an open subset $\Omega\subset \R^d$, we denote by $\mathscr{A}(\Omega)$ the space of real analytic functions defined on $\Omega$. We recall that every $f\in \mathscr{A}(\Omega)$ can be extended holomorphically to a complex neighbourhood $U\subset \C^d$ of $\Omega$, i.e., we can consider $f\in\mathcal{H}(U)$ for some $\Omega\subset U\subset \C^d$ open set. The space $\mathcal{H}(U)$ is endowed with its natural (Fréchet) topology of uniform convergence on compact subsets. Given a compact set $K\subset \C^d$, the space $\mathcal{H}(K)$ of holomorphic germs on $K$ with its natural locally convex topology is
\[
\mathcal{H}(K)=\ind_{n\in\N} \mathcal{H}(U_n),
\]
where $(U_n)_n$ is a basis of $\C^d$-neighbourhoods of $K$. Thus, the space $\mathscr{A}(\Omega)$ has a description as a countable projective limit
\[
\mathscr{A}(\Omega)=\proj_{j\in\N} \mathcal{H}(K_j),
\]
where $(K_j)_j$ is a fundamental sequence of compact subsets of $\Omega$.

Several basic facts about spaces of real analytic functions were studied by Doma\'nski and Vogt \cite{dv}, including the surprising result that this natural space has no basis.

Given a real analytic map $\varphi :\Omega\to \Omega$, the composition operator $C_\varphi :\mathscr{A}(\Omega )\to \mathscr{A}(\Omega )$, $f\mapsto f\circ \varphi$, is continuous. The dynamics of composition operators on spaces of real analytic functions was thoroughly studied in \cite{bd12}. The dynamics of other natural operators, namely weighted backward shifts, on spaces of real analytic functions was recently studied in \cite{dk18}.

Bonet and Doma\'nski \cite{bd12} asked whether the composition operator $C_\varphi$, $\varphi(z):=z^2$, is (sequentially) hypercyclic on $\mathscr{A}(]0,1[)$. They also asked if every sequentially hypercyclic operator $C_\varphi :\mathscr{A}(\Omega)\to \mathscr{A}(\Omega)$ is so that there exists a complex neighbourhood $U$ of $\Omega$ such that $\varphi$ extends holomorphically to $U$, $\varphi (U)\subset U$ and $C_\varphi :\mathcal{H}(U)\to \mathcal{H}(U)$ is hypercyclic. The following example provides a positive answer to the first question, and a negative answer to the second one.

\begin{example} Let $\varphi (z):=z^2$, $X:=\mathscr{ A}(]0,1[)$, and $T:=C_\varphi$. Let $p_n(z)$
be a dense sequence of polynomials. We set $x_n(z)=z(1-z)p_n(z)$, $n\in \N$, which forms
a sequentially dense set in $X$. Let $X_0:=\{ x_n \ ; \ n\in \N\}$ and $Y:=\mathcal{ H}(U)$, for $U$ the open disk centered at $1/2$
of radius $1/2$. Let $\log z$ be a branch of the logarithm defined on $\C \setminus ]-\infty ,0]$.
We set $S_n=C_{\gamma_n}$, where $\gamma_n(z)=\exp (\frac{1}{2^n} \log z)$, $n\in \N$.
It is clear that the conditions of Corollary~\ref{mcor} are satisfied for the sequence of all natural numbers,
and $T$ is sequentially hypercyclic. Indeed,
$T^kS_kf=f$ on $]0,1[$ for all $f\in \mathcal{ H}(U)$ (considered as a subspace of $\mathscr{ A}(]0,1[)$). Since $f\in \mathcal{H}(U)$, we even have $T^kS_kf=f\in \mathcal{H}(U)$ for each $k\in \N$, so that (iii)' is trivially satisfied.
Given a compact set $K\subset U$, $\varphi^n\rightarrow 0$ and $\gamma_n\rightarrow 1$ uniformly on $K$.
Therefore,
$$
\lim_{n\rightarrow \infty}(T^nS_jx_m)(z)=\lim_{n\rightarrow \infty}z^{2^{n-j}}(1-z^{2^{n-j}})p_m(z^{2^{n-j}})= 0,
\ \forall j,m\in \N ,
$$
and
$$
\lim_{n\rightarrow \infty}(T^jS_nx_m)(z)=\lim_{n\rightarrow \infty}\gamma_{n-j}(z)(1-\gamma_{n-j}(z))p_m(\gamma_{n-j}(z))= 0,
\ \forall j\geq 0, \forall m\in \N ,
$$
uniformly on $K$. That is, (i)' and (ii)' are also satisfied.
\end{example}

The idea of the previous example can be extended to certain composition operators $C_\varphi :\mathscr{A}(I)\to\mathscr{A}(I)$ for bounded open intervals $I$ in $\R$. These results will appear elsewhere. We should also note the following alternative argument provided by José Bonet: A classical result of Belitskii and Lyubich \cite{bl99} (see also \cite{bd15}) shows that any real analytic diffeomorphism without fixed points $\varphi:\R\to\R$ is real analytic conjugate to the shift $x\mapsto x + 1$. As a consequence, for any real analytic diffeomorphism without fixed points $\varphi: I \to I$ on an open interval $I\subset \R$, the composition operator $C_\varphi :\mathscr{A}(I)\to\mathscr{A}(I)$  is sequentially hypercyclic.


\begin{example}
The snake shift $T$ constructed in \cite{bfpw} was defined on the countable direct sum $X:=\oplus_{i\in\N} Y$ of a Fréchet sequence space $Y$, for the cases $Y=\ell^p$, $1\leq p<\infty$, $Y=c_0$, $Y=s$, the space of rapidly decreasing sequences
\[
s:=\{ x=(x_n)_n\in \C^\N \ ; \ \norm{x}_k:=\sum_{n\in\N} |x_n|n^k<\infty \mbox{ for all } k\in\N\}.
\]
To fix notation, $(e_{i,j})_j$ represents the canonical unit vectors on the $i$-th summand, $Te_{1,1}=0$, $Te_{i,j}=\lambda e_{f(i,j)}$, $(i,j)\neq (1,1)$, where the constant $\lambda >1$ and $f:\N \times \N\setminus \{(1,1)\}\to \N\times \N$ is a suitable bijection. Once a certain sequence of vectors with finite support $(x_j)_j$ in $X$ was fixed so that it is sequentially dense in $X$, the constructed hypercyclic vector had the form
\[
x=\sum_{k\in\N} \sum_{j=m_k}^{n_k} \frac{1}{\lambda^{l_k}} \alpha_j e_{1,j} \in Y,
\]
where $T^{l_k}(\sum_{j=m_k}^{n_k} \frac{1}{\lambda^{l_k}} \alpha_j e_{1,j})=x_k$ and $|\alpha_j|\leq k$, $m_k\leq j\leq n_k$, $k\in\N$, for suitable increasing sequences $(m_k)_k$, $(n_k)_k$ and $(l_k)_k$. In the case $Y=s$, the sequence $(n_k)_k$ was required to be polynomially bounded (actually, $n_k\leq 3k^2$, $k\in\N$).

Defining $X_0=\{ x_k \ ; \ k\in\N\}$ and $Se_{i,j}=\lambda^{-1}e_{f^{-1}(i,j)}$, $(i,j)\in\N\times\N$, $S_n=S^n$, $n\in \N$, one has
\[
S_{l_k}x_k=\sum_{j=m_k}^{n_k} \frac{1}{\lambda^{l_k}} \alpha_j e_{1,j}, \ \ k\in\N, \mbox{ that yields condition (iii)' in Corollary \ref{mcor}},
\]
\[
T^{l_k}S_{l_j}x_i =0, \mbox{ if } k>j+ i,
\]
that is, condition (i)' in Corollary \ref{mcor}, and finally
\[
T^{l_j}S_{l_k}x_i \in Y, \mbox{ if } k>j+ i, \ \mbox{ and } \ \lim_k T^{l_j}S_{l_k}x_i =\lim_k \frac{1}{\lambda^{l_k-l_j}} \sum_{r=m_k-l_j}^{m_k-l_j+n_i-m_i} \beta_r e_{1,r}=0 \mbox{ in } Y,
\]
for certain $\beta_r$ with $|\beta_r|\leq i$, for $m_k-l_j\leq r\leq m_k-l_j+n_i-m_i$, $k>i+j$, which gives condition (ii)' in Corollary \ref{mcor}.
\end{example}

We want to point out that Shkarin constructed in \cite{shk} hypercyclic operators on locally convex direct sums of sequences $(X_n)_n$ of separable Fréchet spaces for which infinitely many of them are infinite dimensional, and he characterized inductive limits of sequences of separable Banach spaces which support a hypercyclic operator.

\section*{Acknowledgements}
The author thanks José Bonet and Pawe{\l} Doma\'nski for interesting conversations on the results of the paper, and the referee for valuable comments that produced an improved presentation. He also acknowledges  the support of MINECO,  Project MTM2016-75963-P, and Generalitat Valenciana, Project  PROMETEO/2017/102.



\vspace{5mm}

 Institut Universitari de Matemàtica Pura i Aplicada,

 Universitat Polit\`ecnica de Val\`encia,

 Edifici 8E, Acces F, 4a planta,

 46022 Val\`encia, Spain.

\end{document}